\documentclass[10pt]{olplainarticle}
\usepackage{setspace}
\usepackage[inkscapelatex=false]{svg}
\title{Adaptive Hybrid Subspace Levenberg–Marquardt Algorithm with Adequacy Monitor for Large-Scale Least Squares Problems}

\author{M. Duc Hoang \& Timothy J. Lewis}
\affil{Department of Mathematics, University of California, Davis, Davis, CA, USA}


\begin{abstract}
The Levenberg–Marquardt (LM) algorithm is the most widely used method for solving nonlinear least-squares problems, as it combines the robustness of steepest descent with the fast local convergence of the Gauss-Newton method. However, its computational cost can become prohibitive for large-scale problems because each iteration requires solving a large damped linear system, and conventional step-acceptance strategies may require repeated solves as the damping parameter is adjusted. Despite this computational challenge, many large-scale least-squares problems exhibit effective low-dimensional structure, with only a small number of parameter-space directions strongly informed by the data. We propose an adaptive hybrid subspace Levenberg–Marquardt (HSLM) algorithm that constructs a low-dimensional subspace from complementary sources of gradient, memory, Krylov-subspace, and randomized curvature information and computes a spectrally damped LM step within this subspace. A distinguishing feature of the method is a deterministic adequacy monitor that quantifies how much descent information is captured by the reduced space and adaptively enriches the subspace when necessary.  Step acceptance is decoupled from damping adjustment: Armijo backtracking determines the accepted step length, while the ratio of actual to predicted reduction is used solely to update the damping parameter, thereby avoiding repeated damped-system solves during step acceptance. For the HSLM algorithm, we establish global convergence to stationarity and prove local linear and superlinear convergence. Numerical experiments on neural-network training problems show that HSLM achieves convergence behavior comparable to classical and Krylov-subspace LM (KSLM) while substantially reducing per-iteration computational cost, with increasing advantages observed as the parameter dimension grows.
\end{abstract}
\begin{document}
\flushbottom
\maketitle
\thispagestyle{empty}
\section{Introduction}
Large-scale nonlinear least-squares problems arise in many areas, including physics, engineering, inverse modeling, and machine learning. Second-order methods, such as Levenberg--Marquardt (LM) algorithm, are attractive because of their robustness and fast local convergence, especially for ill-conditioned nonlinear least-squares problems \cite{Levenberg1944,Marquardt1963,MadsenNielsenTingleff2004}. However, their direct application becomes computationally expensive in high-dimensional problems, since each iteration requires solving large linear systems involving the Jacobian or an approximate Hessian \cite{NocedalWright2006,DennisSchnabel1996}. At the same time, many inverse problems possess intrinsic low-dimensional structure, in the sense that the data constrain only a relatively small subspace of the full parameter space, while directions outside this subspace are weakly identifiable or dominated by regularization. This observation is consistent with active-subspace ideas, which identify important directions along which the model or objective varies most strongly \cite{constantine2015active,constantine2016accelerating}. This structure motivates reduced-space and low-rank variants of LM, which seek to approximate the full problem within a lower-dimensional subspace that captures the dominant data-informed and curvature directions.

Subspace and low-rank approximations have long been studied as a way to reduce the cost of LM-type methods for large-scale inverse and nonlinear least-squares problems. These approaches include modified LM methods for large-scale inverse problems \cite{VogelWade1993}, general subspace frameworks for nonlinear equations and nonlinear least squares \cite{yuan2009subspace}, randomized truncated SVD approximations of sensitivity matrices \cite{Bjarkason2018}, and Krylov-based projection methods for highly parameterized inverse models \cite{Lin2016}. In parallel, random embedding and randomized subspace optimization methods have been developed to reduce the dimension of the variable space itself \cite{Cartis2022,Shao2021}. Related randomized techniques include stochastic Hessian approximation, as in SPAN \cite{huang2020span}, and randomized matrix-decomposition methods for identifying dominant low-rank subspaces \cite{halko2011finding}. Together, these works show that subspace approximations provide a natural route toward scalable LM-type methods by reducing the cost of each iteration while retaining important curvature and data-informed directions.

Despite these advances, most existing subspace methods lack an explicit deterministic mechanism to ensure that the chosen subspace remains adequate as the optimization progresses and local curvature directions change. Randomized subspace methods, for example, rely on probabilistic embedding properties to guarantee sufficient descent or curvature capture with high probability, but they do not necessarily exploit problem-specific curvature structure or incorporate memory from previous iterations. In addition, low-rank and Krylov-based approaches can efficiently approximate individual LM steps, but the quality of the reduced space may depend strongly on how well the current subspace captures the dominant local geometry. As a result, achieving both robust global convergence and fast local convergence within a scalable, matrix-free framework remains challenging.

In this work, we propose a hybrid subspace Levenberg-Marquardt (HSLM) algorithm that integrates ideas from spectral truncation, Krylov subspace methods, and randomized subspace methods. At each iteration, the method constructs a compact subspace that explicitly incorporates descent information via the gradient, previous accepted steps, and curvature information obtained from Hessian or Jacobian-vector probes. A projected-gradient adequacy measure is employed to monitor descent-relevant information in the subspace and to adaptively expand or refine it when necessary. To avoid repeatedly solving expensive damped systems, we also use Armijo backtracking line search for step acceptance and rely exclusively on the predicted-versus-actual reduction ratio to update the damping parameter. Under standard smoothness, regularity, and vanishing-approximation assumptions, we establish global convergence of the HSLM algorithm, with global convergence to stationarity and local linear/superlinear convergence rates. Finally, we implement the HSLM algorithm to train feedforward neural networks of varying problem sizes and compare its performance with the classical LM and Krylov subspace LM algorithms.

\section{Background and notation}
Nonlinear least-squares problems arise widely in parameter estimation, scientific computing, engineering, and machine learning, and are commonly formulated as
\[
 \min_{x\in \mathbf{R}^n} \mathbf{F}(x) := \min_{x\in \mathbf{R}^n} \dfrac{1}{2} \|\mathbf{r}(x)\|^2,
\]
where $\mathbf{r}(x):\mathbf{R}^n\to\mathbf{R}^m$ is a vector of the residual. The Gauss–Newton method \cite{DennisSchnabel1996} solves nonlinear least-squares problems by linearizing the residual function at each step. At iteration $k$, the residual is approximated by $\mathbf{r}(x_k + s_k) \approx \mathbf{r}(x_k) + \mathbf{J}(x_k) s_k$, leading to a linear least-squares subproblem where the update direction $s_k$ satisfies the equations
\[
\left(\mathbf{J}(x_k)^\top \mathbf{J}(x_k)\right)s_k = -\mathbf{J}(x_k)^\top\mathbf{r}(x_k)
\]
where $\mathbf{J}(x)\in\mathbf{R}^{m\times n}$ is the Jacobian of the residuals and $\mathbf{H}(x_k)\approx\mathbf{J}(x_k)^\top \mathbf{J}(x_k)$ is the corresponding Gauss–Newton approximation to the Hessian matrix. Gauss–Newton exhibits second-order local convergence when residuals are small and the Jacobian has full rank. However, the method may fail when the Jacobian is ill-conditioned or the initial guess is far from the optimal solution. To improve robustness, Levenberg \cite{Levenberg1944} and Marquardt \cite{Marquardt1963} introduced a damping strategy that modifies the system as
\begin{equation}\label{eq:LM}
\left(\mathbf{J}(x_k)^\top\mathbf{J}(x_k) + \mu_k \mathbf{I}_n\right)s_k = -\mathbf{J}(x_k)^\top\mathbf{r}(x_k)
\end{equation}
where the damping term $\mu_k> 0$. In the classical Levenberg-Marquardt algorithm, the damping update strategy plays a key role in the method's stability and efficiency. When the residual is effectively reduced, the damping parameter is decreased, bringing the algorithm closer to the Gauss-Newton direction. Conversely, when an iteration produces insufficient reduction in the residual, the damping parameter is increased, yielding a step closer to the steepest-descent direction.\\
For notational convenience, we denote 
\[\mathbf{g}(x)=\nabla \mathbf{F}(x)=\mathbf{J}(x)^\top \mathbf{r}(x)\] 
and denote the residual, the Jacobian matrix, the gradient and Gauss-Newton Hessian approximation at $x_k$ by 
\[r_k=\mathbf{r}(x_k),\hspace{1em} J_k=\mathbf{J}(x_k),\hspace{1em} g_k=\mathbf{g}(x_k),\hspace{1em} H_k=\mathbf{H}(x_k).\]
\section{Hybrid Subspace Levenberg Marquart Algorithm}
The HSLM algorithm is based on the observation that, although the optimization takes place in a high-dimensional parameter space, only a relatively small number of directions are typically needed to capture the dominant descent and curvature information. Rather than relying on a fixed reduced space, HSLM constructs an adaptive hybrid subspace at each iteration from multiple sources of information, including the gradient, recent optimization history, Krylov/Lanczos directions, and randomized curvature probes. A distinguishing feature of the method is a deterministic adequacy monitor that evaluates whether the current subspace captures sufficient descent information and adaptively enriches the subspace whenever necessary. Once an adequate subspace has been constructed, the algorithm computes a spectrally damped Levenberg–Marquardt step within the reduced space. Step acceptance is determined using Armijo backtracking, while the ratio of actual to predicted reduction is used solely to update the damping parameter, thereby avoiding repeated damped-system solves. The remainder of this section describes the subspace construction, the adequacy monitor, the reduced LM system, step acceptance and damping update.

\subsection{Subspace construction}
At each iteration, HSLM constructs an orthonormal basis $V_k\in\mathbb{R}^{n\times p}$ from candidate directions that provide complementary information about the local optimization problem. Krylov directions provide an efficient way to incorporate local curvature information \cite{Vinyals2012,yuan2009subspace}, while randomized Hessian-vector products provide an inexpensive way to sample local curvature information without explicitly forming the Hessian \cite{huang2020span}. Other dimension-reduction techniques include randomized subspace and embedding methods that construct lower-dimensional representations of high-dimensional optimization problems while approximately preserving relevant geometric information \cite{Cartis2022,Shao2021}, and truncated SVD approaches that approximate dominant directions of the sensitivity matrix \cite{Bjarkason2018}. Motivated by these complementary ideas, HSLM draws candidate and enrichment directions from four complementary sources:
\begin{itemize}
\item the normalized gradient, which provides first-order descent information; 
\item recent optimization steps, which encode local optimization history; 
\item Krylov/Lanczos vectors, which approximate dominant curvature directions; and 
\item randomized Gauss-Newton-vector probes $(J_k^\top J_k) w$ where $w\sim \mathcal{N}(\mathbf{0},\mathbf{I}_n)$, which broaden spectral coverage and improve robustness when the local curvature is noisy or heterogeneous. 
\end{itemize}
In this way, the method adapts the reduced space to the local geometry at each iteration while keeping the dominant linear algebra restricted to a low-dimensional subproblem.

\subsection{Deterministic Adequacy Monitor}
Although the candidate pool combines multiple complementary sources of information, there is no guarantee that the resulting reduced space adequately represents the current descent geometry. Therefore, before computing the reduced Levenberg–Marquardt step, we evaluate the quality of the subspace using a deterministic adequacy criterion. Specifically, we measure the fraction of the gradient energy captured by the current subspace through the projected-gradient energy ratio
\begin{equation}\label{eq:eta_sub}
\eta^{\mathrm{sub}}_k := \frac{\|V_k^\top g_k\|^2}{\|g_k\|^2}\in[0,1].
\end{equation}
The quantity $\eta^{\mathrm{sub}}_k$ takes values in $[0,1]$, with larger values indicating that the subspace captures a greater proportion of the full gradient and therefore contains sufficient first-order descent information. If $\eta^{\mathrm{sub}}_k$ drops below a prescribed threshold $\eta_{\min}\in(0,1]$, additional candidate directions are generated from the pool described in Section 3.1, the basis is updated, and the adequacy test is repeated until the threshold is satisfied.

\subsection{Reduced LM system}
Once the adequacy criterion is satisfied, the reduced basis is fixed for the current iteration. The Levenberg–Marquardt step is then computed entirely within this subspace, reducing the n-dimensional damped system to a much smaller p-dimensional problem. Restricting the Jacobian to the reduced subspace gives the reduced Jacobian $J_{k,p} := J_k V_k \in \R^{m\times p}$.\\
Let 
\begin{equation}\label{eq:svd}
J_{k,p} = U_k\Sigma_{k,p} Z_k^\top,
\qquad
\Sigma_{k,p}=\mathrm{diag}(\sigma_{k,1},\dots,\sigma_{k,p}).
\end{equation}
denote its singular value decomposition. This decomposition diagonalizes the reduced Gauss–Newton matrix restricted to the reduced subspace and provides a convenient basis for constructing the reduced LM system. In particular,
\begin{equation}\label{eq:6}
V_k^\top H_k V_k = (J_kV_k)^\top(J_kV_k)=J_{k,p}^\top J_{k,p} = Z_k\Sigma_{k,p}^2 Z_k^\top,
\end{equation}
and
\begin{equation}\label{eq:gred}
V_k^\top g_k = V_k^\top (J_k^\top r_k) = (J_kV_k)^\top r_k = J_{k,p}^\top r_k = Z_k\Sigma_{k,p}U_k^\top r_k.
\end{equation}
We restrict the search direction $s_k$ to the reduced subspace by writing $s_k = V_k Z_k y_k$. Using \eqref{eq:6} and \eqref{eq:gred}, the classical reduced Levenberg–Marquardt system becomes
\[
\left( \Sigma_{k,p}^2 + \mu_k \mathbf{I}_p\right) y_k = - \Sigma_{k,p} U_k^\top r_k
\]
This reduced formulation inherits the isotropic damping of the standard LM algorithm. However, because the singular values of the reduced Jacobian can vary substantially, treating all directions equally may lead to unnecessary damping in some directions and insufficient damping in others. To better accommodate ill-conditioning and varying local curvature, we replace the isotropic damping matrix $\mathbf{I}_p$ with a diagonal, curvature-dependent damping matrix $D_{k,p}$, giving
\begin{equation}\label{eq:reducedLM}
 B_k y_k = - \Sigma_{k,p} U_k^\top r_k  \qquad \text{where } B_k := \Sigma_{k,p}^2 + \mu_k D_{k,p},
\end{equation}
and 
\begin{equation}\label{eq:D}
D_{k,p} := \mathrm{diag}(d_{k,1},\dots,d_{k,p}),\qquad d_{k,i} := \max(\sigma_{k,i}^2,\delta).
\end{equation}
The threshold $\delta>0$ guarantees that every damping coefficient remains strictly positive, preventing excessively large updates associated with very small singular values. Consequently, the amount of regularization adapts to the local curvature represented by the singular values: directions associated with larger singular values receive stronger damping, while flatter directions remain comparatively less damped. This curvature-adaptive damping improves numerical stability while preserving efficient progress along weakly constrained directions. The resulting reduced system \eqref{eq:reducedLM} is therefore solved in the low-dimensional subspace, after which the corresponding search direction is mapped back to the original parameter space.
\begin{remark}
When the candidate subspace becomes large, its dimension may be controlled by compressing the basis using an SVD-based energy criterion. In practice, the basis can be maintained efficiently using incremental QR factorization together with an SVD of a small compressed matrix, following standard incremental SVD techniques \cite{Brand2002incremental}. 
\end{remark}
\subsection{Step acceptance and damping updates}
After computing the reduced LM search direction $s_k$, we determine an appropriate step length using Armijo backtracking \cite{Armijo1966}. Rather than accepting the full step unconditionally, Armijo backtracking performs a line search to ensure a sufficient decrease in the objective function. Specifically, Armijo backtracking starts with the full step $(t_k=1)$ and repeatedly reduces the step length by a factor $\beta \in (0,1)$ until the Armijo condition
\begin{equation}\label{eq:armijo}
F(x_k+t_k s_k)\le F(x_k)+\alpha\, t_k\, g_k^\top s_k,
\qquad \alpha\in (0,1),
\end{equation}
is satisfied. The accepted step length is therefore $t_k=\beta^{j_k}$, where $j_k$ is the smallest nonnegative integer for which the inequality holds. The iterate is updated according to 
\[x_{k+1}=x_k+t_k s_k.\]
Once a step has been accepted, we define the damped quadratic model associated with the reduced Levenberg–Marquardt system,
\begin{equation}\label{eq:model}
m_k(s):=\dfrac{1}{2}\norm{r_k+J_k s}^2 + \frac{\mu_k}{2}\norm{D_{k,p}^{1/2}Z_k^\top V_k^\top s}^2.
\end{equation}
This model combines the linearized least-squares objective with the curvature-adaptive damping introduced in Section 3.3 and provides the predicted reduction used to update the damping parameter. The agreement between the predicted reduction of the model and the actual reduction of the objective is measured by the gain ratio
\begin{equation}\label{eq:rho}
\rho_k := \frac{F(x_k)-F(x_k+t_k s_k)}{m_k(0)-m_k(t_k s_k)}.
\end{equation}
Following \citep{MadsenNielsenTingleff2004,ConnGouldToint2000}, the damping parameter is updated according to the gain ratio using the standard rule
\begin{equation}\label{eq:lambdaup}
\mu_{k+1}=
\begin{cases}
\mu_k/\mu_{down}, & \rho_k\ge \rho_{\mathrm{high}}\\
\mu_k, & \rho_{\mathrm{low}}\le \rho_k < \rho_{\mathrm{high}}\\
\mu_k.\mu_{up}, & \rho_k<\rho_{\mathrm{low}}
\end{cases}.
\end{equation}
Unlike the classical LM algorithm, step acceptance and damping updates are decoupled: Armijo backtracking determines the step length, while the gain ratio is used to update the damping parameter. This separation avoids repeated damped-system solves during step acceptance while preserving the role of the damping parameter as a measure of local quadratic approximation quality.

\subsection{Pseudocode}
\begin{algorithm}[H]
\caption*{\textbf{Algorithm} Hybrid subspace Levenberg Marquardt}
\label{alg:main}
\begin{algorithmic}
\STATE \textbf{Input:} Initial guess $x_0$, Jacobian matrix $J_k$, gradient vector $g_k$, damping term $\mu_0>0$, singular value threshold $\delta>0$, adequacy threshold $\eta_{\min}\in(0,1]$, Armijo constants $(\alpha,\beta)$, ratio thresholds $(\rho_{\mathrm{low}},\rho_{\mathrm{high}})$.
\FOR{$k=0,1,2,\dots$}
\STATE Construct a candidate basis, then orthonormalize it to obtain $V_k$.
\STATE Compute $\eta^{\mathrm{sub}}_k=\|V_k^\top g_k\|^2/\|g_k\|^2$.
\WHILE{$g_k\neq 0$ \textbf{and} $\eta^{\mathrm{sub}}_k<\eta_{\min}$}
  \STATE Expand the subspace, then update $V_k$.
  \STATE Recompute $\eta^{\mathrm{sub}}_k$.
\ENDWHILE
\STATE Form $J_{k,p}=J_kV_k$ and its SVD $J_{k,p}=U_k\Sigma_{k,p}Z_k^\top$.
\STATE Set $D_{k,p}=\mathrm{diag}(\max(\sigma_{k,i}^2,\delta))$ and $B_k=\Sigma_{k,p}^2+\mu_k D_{k,p}$.
\STATE Solve $B_k y_k = -\Sigma_{k,p}U_k^\top r_k$ and set $s_k=V_k Z_k y_k$.
\STATE Find smallest $j\ge 0$ such that $t_k=\beta^j$ satisfies \eqref{eq:armijo}; set $x_{k+1}=x_k+t_k s_k$.
\STATE Update $\mu_{k+1}$.
\ENDFOR
\end{algorithmic}
\end{algorithm}
\section{Convergence analysis}
This section establishes the global and local convergence properties of HSLM. The analysis follows the standard framework for line-search Levenberg–Marquardt methods while accounting for the adaptive reduced subspace. Building on the classical analysis, we identify the additional conditions needed to ensure that restricting the LM step to an adaptive subspace preserves the relevant global and local convergence properties. For global convergence, the key ingredients are the deterministic adequacy condition, which ensures that the reduced space retains sufficient gradient information, and the Armijo line search, which guarantees sufficient decrease. For local convergence, we quantify how accurately the reduced system approximates the full Newton system and show how the resulting approximation errors determine the local convergence rate.

\subsection{Global Convergence}

We first impose two standard assumptions used in the global convergence analysis of line-search methods ~\cite{Armijo1966,NocedalWright2006}. Lower boundedness prevents indefinite decrease of the objective along the initial level set, $\mathcal{L}:=\{x: F(x)\le F(x_0)\}$, while Lipschitz continuity of the gradient provides the smoothness needed to establish sufficient decrease.

\begin{assumption}[Lower boundedness on the level set]\label{ass:A1}
$F$ is bounded below on $\mathcal{L}$.
\end{assumption}

\begin{assumption}[Lipschitz gradient on the level set]\label{ass:A2}
$\nabla F$ is Lipschitz continuous on an open neighborhood of $\mathcal{L}$: there exist $L_1>0$ such that $\norm{\nabla F(x)-\nabla F(y)}\le L_1\norm{x-y}$ for all $x,y\in\mathcal{L}$.
\end{assumption}
The remaining assumptions account for the fact that the search direction is computed only in a reduced subspace. In a full-space method, a nonzero gradient directly provides a descent direction. In a subspace method, however, the projected gradient may be arbitrarily small if the chosen subspace is poorly aligned with the gradient. Therefore, we required the algorithm to construct or enrich the subspace until it captures a uniformly positive fraction of the full-gradient information. This deterministic adequacy condition is analogous to alignment and subspace-quality requirements appear in randomized and reduced subspace methods ~\cite{Cartis2022,Shao2021}.
\begin{assumption}[Deterministic adequacy enforcement]\label{ass:A3}
There exists $\eta_{\min}\in(0,1]$ such that the algorithm enforces $\eta^{\mathrm{sub}}_k\ge \eta_{\min}$ for all iterations with $g_k\neq 0$, where $\eta^{\mathrm{sub}}_k$ is defined in \eqref{eq:eta_sub}.
\end{assumption}
Finally, we assume the reduced problem to remain uniformly well conditioned. Positive definiteness of the reduced matrix guarantees that the reduced model is strictly convex and that the subspace step is uniquely defined, while uniform eigenvalue bounds prevent the reduced system from becoming nearly singular or excessively scaled. These bounds also provide the estimates needed to relate the step norm and predicted reduction to the projected gradient, and are standard conditioning requirements for subspace methods ~\cite{DennisSchnabel1996,MadsenNielsenTingleff2004}, even when the reduced Jacobian is rank deficient.
\begin{assumption}[Uniform boundedness of reduced conditioning]\label{ass:A4}
There exists $m_B,M_B$ such that $0<m_B\le \lambda_{\min}(B_k)\le \lambda_{\max}(B_k)\le M_B<\infty$ for all $k$, where $B_k$ is defined in \eqref{eq:reducedLM}.
\end{assumption}

Before establishing descent and convergence properties, we first verify that the reduced Levenberg–Marquardt system is well posed. The following lemma shows that the regularization term makes the reduced matrix positive definite whenever $\mu_k>0$, ensuring that the reduced step exists uniquely.
\begin{lemma}[Positive definiteness of the reduced system]\label{lem:spd}
$B_k=\Sigma_{k,p}^2+\mu_k D_{k,p}$ is symmetric positive definite.
\end{lemma}
\begin{proof}
Since $\Sigma_{k,p}^2\succeq 0$, $\mu_k>0$ and \eqref{eq:D}, hence $B_k\succ 0$.
\end{proof}
The next result shows that the positive definiteness of the reduced system translates into descent for the full objective. In particular, the deterministic adequacy condition prevents the projected gradient from vanishing when $g_k\neq0$, while $B_k\succ0$ ensures that the reduced LM solution produces a strictly negative directional derivative.
\begin{lemma}[Descent direction]\label{lem:descent}
Under Assumption \ref{ass:A3}, if $g_k\neq 0$, then the reduced LM step $s_k$ defined in \eqref{eq:reducedLM} satisfies $g_k^\top s_k<0$.
\end{lemma}
\begin{proof}
From \eqref{eq:gred} and \eqref{eq:reducedLM},
\[
\begin{aligned}
g_k^\top s_k 
&= g_k^\top V_k Z_ky_k\\[4pt]
&= (V_k^\top g_k)^\top  Z_ky_k\\[4pt]
&= (Z_k\Sigma_{k,p}U_k^\top r_k)^\top  Z_ky_k\\[4pt]
&= (\Sigma_{k,p}U_k^\top r_k)^\top  y_k\\[4pt]
&= -(\Sigma_{k,p}U_k^\top r_k)^\top  B_k^{-1}\Sigma_{k,p}U_k^\top r_k.
\end{aligned}
\]
By Lemma~\ref{lem:spd}, $B_k^{-1}\succ 0$, hence $g_k^\top s_k<0$ whenever $g_k\neq 0$.
\end{proof}
After showing $s_k$ is a strict descent direction, we next verify that the line-search procedure is well defined. By Lipschitz continuity of the gradient, the first-order decrease along $s_k$ dominates the quadratic Taylor remainder for sufficiently small step lengths, ensuring that the Armijo condition is eventually satisfied.
\begin{lemma}[Existence of an Armijo step]\label{lem:armijo}
Assume Assumption~\ref{ass:A2}. If $g_k^\top s_k<0$ and $\alpha \in (0,1)$, then there exists $\bar t_k>0$ such that for all $t\in(0,\bar t_k]$,
\[
F(x_k+t s_k)\le F(x_k)+\alpha t\, g_k^\top s_k,
\]
and Armijo backtracking terminates.
\end{lemma}
\begin{proof}
From Assumption~\ref{ass:A2}, we have
\[
\begin{aligned}
F(x_k + t s_k) - F(x_k)
&= \int_{0}^{1} \nabla F\!\bigl(x_k + y t s_k\bigr)^{\top} (t s_k)\, dy \qquad\qquad \text{(By Fundamental Theorem of Calculus)}\\[4pt]
&= \int_{0}^{1} \nabla F(x_k)^{\top} (t s_k)\, dy
  + \int_{0}^{1} \Bigl(\nabla F\!\bigl(x_k + y t s_k\bigr) - \nabla F(x_k)\Bigr)^{\top} (t s_k)\, dy \\[4pt]
&\le t\, g_k^{\top} s_k
  + \int_{0}^{1} L_1 \,\|y t s_k\|\, \|t s_k\|\, dy \\[4pt]
&\le \ t\, g_k^{\top} s_k
  + L_1t^2\|s_k\|^{2}\int_{0}^{1} y dy \\[4pt]
&\le t\, g_k^{\top} s_k + \frac{L_1}{2}\, t^{2}\, \|s_k\|^{2}.
\end{aligned}
\]
Then,
\[
F(x_k + t s_k) - F(x_k) - \alpha tg_k^\top s_k \le (1-\alpha)t g_k^\top s_k+\frac{L_1}{2}t^2\|s_k\|^2.
\]
Since $g_k^\top s_k<0$, it suffices to choose
\[
0<t\le \bar t_k := -\frac{2(1-\alpha)g_k^\top s_k}{L_1\|s_k\|^2}.
\]
Backtracking over $t=\beta^{j_k}$ will eventually produce $t\le \bar t_k$ for some $j_k$.
\end{proof}
We now combine the preceding lemmas to establish global convergence. Specifically, we show that every infinite sequence generated by the algorithm has a subsequence along which the gradient norm converges to zero.
\begin{theorem}[Global convergence to stationarity]\label{thm:global}
Under Assumptions~\ref{ass:A1}, \ref{ass:A2}, \ref{ass:A3}, and \ref{ass:A4}, either the algorithm terminates at a stationary point, or
\[
\liminf_{k\to\infty}\|g_k\|=0.
\]
\end{theorem}
\begin{proof}
Assume the algorithm does not terminate and suppose, for contradiction, that there exists $\varepsilon>0$ such that $\|g_k\|\ge \varepsilon$ for all $k$. Lemma~\ref{lem:armijo} shows Armijo backtracking holds for all $t\le \bar t_k$ with
\[
\bar t_k=-\frac{2(1-\alpha)g_k^\top s_k}{L_1\|s_k\|^2}.
\]
To prevent $\bar t_k\to 0$, we want to bound $\|s_k\|$ and $g_k^\top s_k$.\\
Consider $y_k=-B_k^{-1}\Sigma_{k,p}U_k^\top r_k$ and $\|s_k\|=\|V_kZ_k y_k\|=\|y_k\|$,
\[
\begin{aligned}
\|s_k\| 
&\le \|B_k^{-1}\|\,\|\Sigma_{k,p}U_k^\top r_k\|  \\[4pt]
&\le \frac{1}{\lambda_{\min}(B_k)}\|\Sigma_{k,p}U_k^\top r_k\| \\[4pt]
&= \frac{1}{\lambda_{\min}(B_k)}\|Z_{k}\Sigma_{k,p}U_k^\top r_k\| \\[4pt]
&= \frac{1}{\lambda_{\min}(B_k)}\|V_k^\top g_k\|.    
\end{aligned}
\]
From Lemma~\ref{lem:descent} and \eqref{eq:gred},
\[
\begin{aligned}
-g_k^\top s_k 
&= (\Sigma_{k,p}U_k^\top r_k)^\top  B_k^{-1}\Sigma_{k,p}U_k^\top r_k \\[4pt]
&\ge \frac{1}{\lambda_{\max}(B_k)}\|\Sigma_{k,p}U_k^\top r_k\|^2 \\[4pt]
&= \frac{1}{\lambda_{\max}(B_k)}\|Z_{k}\Sigma_{k,p}U_k^\top r_k\|^2 \\[4pt]
&= \frac{1}{\lambda_{\max}(B_k)}\|V_k^\top g_k\|^2.    
\end{aligned}
\]
By Assumption~\ref{ass:A4} and ~\eqref{eq:eta_sub},
\[
-g_k^\top s_k
\ge \frac{1}{M_B}\,\eta_{\min}\,\|g_k\|^2
\ge \frac{\eta_{\min}}{M_B}\varepsilon^2
>0.
\]
Then,
\[
\frac{-g_k^\top s_k}{\|s_k\|^2}
\ge \frac{\lambda_{\min}(B_k)^2}{\lambda_{\max}(B_k)}.
\]
From assumption~\ref{ass:A4}, we have $\bar t_k\ge \bar t =\dfrac{2(1-\alpha)m_B^2}{M_BL_1}>0$.\\
Armijo backtracking then returns $t_k\ge  \min\{1,\beta\bar t\}$, which satisfy
\[
F(x_{k+1}) \le F(x_k)+\alpha t_k g_k^\top s_k
\le F(x_k) - \alpha t_k \frac{\eta_{\min}}{M_B}\varepsilon^2.
\]
Thus, $F(x_k)\to -\infty$ as $k \to \infty$, contradicting Assumption~\ref{ass:A1}. Therefore $\liminf_{k\to\infty}\|g_k\|=0$.
\end{proof}

\subsection{Local Convergence}

Now, we introduce the additional assumptions required for the local convergence analysis. Unlike the global assumptions, which ensure sufficient descent and convergence toward stationarity, these conditions describe the local geometry of the objective and require the subspace and curvature approximations to become increasingly accurate as the iterates approach a solution. The key identity for local analysis is the exact Hessian decomposition 
\[ \nabla^2 F(x)=J(x)^\top J(x)\;+S(x) \quad \text{ where } S(x) = \sum_{i=1}^m r_i(x)\,\nabla^2 r_i(x) \] 
where $J(x)^\top J(x)$ is the Gauss–Newton curvature and $S(x)$ collects residual-dependent second-derivative. The Levenberg–Marquardt algorithm is built on the Gauss–Newton approximation $H(x) \approx J^TJ$ and computes steps from a damped system. In this study, we analyze the local convergence of the proposed algorithm under the assumption of zero (or very small) residual at the optimal solutions.

We first impose the standard local regularity conditions used in Newton-type convergence analyses ~\cite{DennisSchnabel1996,NocedalWright2006}. Positive definiteness of the Hessian at $x^\star$ ensures that the solution is locally isolated and that the objective has stable quadratic curvature near the minimizer. Lipschitz continuity of the Hessian controls the remainder in the second-order Taylor expansion and is essential for deriving a local error recursion.
\begin{assumption}[Strict local minimizer and Lipschitz Hessian]\label{ass:L1}
There exists a strict local minimizer $x^\star$ such that $g(x^\star)=0$ and $\nabla^2F(x^\star)\succ 0$.
Moreover, $\nabla^2F(x)$ is Lipschitz in a neighborhood $U$ of $x^\star$:
$\|\nabla^2F(x)-\nabla^2F(y)\|\le L_2\|x-y\|$ for all $x,y\in U$.
\end{assumption}
The next assumption accounts for the evolving reduced subspace, which arise in reduced and randomized subspace methods ~\cite{Cartis2022,Shao2021}. Global convergence requires only that the subspace retain a fixed amount of gradient information, whereas fast local convergence requires the subspace approximation to improve as $x_k$ approaches $x^\star$. Accordingly, the component of the gradient omitted by the subspace must vanish, the reduced space must become increasingly invariant under the local curvature operator, and the LM regularization must disappear as the subspace approximation becomes more accurate.
\begin{assumption}[Vanishing subspace/curvature errors and regularization]\label{ass:L2}
Let $P_k:=V_k V^\top_k$ and define $\gamma_k := \|(I-P_k)g_k\|/\|g_k\|$. Suppose $x_k\to x^\star$, $\gamma_k\to 0$, $\mu_k\to 0$, and there exists $\varepsilon_k\to 0$ such that
\[
\|H_kP_k - P_kH_kP_k\|\le \varepsilon_k\|H_k\|.
\]
\end{assumption}
Here, $\gamma_k$ measures the relative portion of the gradient missed by the subspace, while $\epsilon_k$ measures the relative curvature (Hessian) information not captured by the subspace. As $\gamma_k,\epsilon_k \to 0$, the reduced subspace increasingly captures both the relevant gradient and curvature information near $x^\ast$. Finally, the local curvature approximation used in the LM system must become sufficiently accurate near the solution. The following condition requires the Hessian approximation error to decrease at least linearly with the distance to $x^\star$, which ensures that the Hessian approximation error becomes negligible at the rate required for fast local convergence. Consequently, its contribution to the local iterate-error recursion is of quadratic order, allowing the method to recover the fast local behavior of a Newton-type iteration.
\begin{assumption}[Vanishing Hessian approximation error near $x^\star$]\label{ass:L3}
Assume there exist $c>0$ such that for all $x$ close enough to $x^\star$,
\[
\|S(x)\|\le c\|x-x^\star\|.
\]
\end{assumption}

To analyze the local convergence rate, we next quantify how accurately the reduced regularized step satisfies the full Newton equation. The following lemma provides a bound on the step norm and decomposes the exact Newton residual into four distinct error components: the gradient component omitted by the subspace, the lack of curvature invariance of the subspace, the LM regularization, and the Hessian-approximation error. This decomposition makes explicit the conditions under which the reduced LM step becomes an increasingly accurate inexact Newton step near $x^\star$.
\begin{lemma}[Exact Newton residual bound]\label{lem:inexact}
Let $P_k:=V_kV_k^\top$ and suppose $s_k=V_k Z_ky_k$ satisfies a reduced regularized system
\[
(Z_k^\top V_k^\top H_k V_kZ_k + R_k) y_k = -Z_k^\top V_k^\top g_k,
\]
where $R_k = \mu_kD_{k,p}\succeq 0$ and $Z_k^\top V_k^\top H_k V_k Z_k + R_k \succeq m_BI$ for some $m_B>0$.
Then:
\begin{enumerate}
\item $\|s_k\| = \|y_k\|\le \dfrac{1}{m_B}\|V_k^\top g_k\|\le\dfrac{1}{m_B}\|g_k\|$.
\item Defining the exact Newton residual $R_k^{exact}:=\nabla^2F(x_k) s_k + g_k$,
\[
\|R_k^{exact}\|
\le \|(I-P_k)g_k\| + \|H_kP_k - P_kH_kP_k\|\,\|s_k\| + \|R_k\|\,\|y_k\| +\|S_k\|\|s_k\|.
\]
\end{enumerate}
\end{lemma}
\begin{proof}

(1) Let $B_k:=Z_k^\top V_k^\top H_k V_k Z_k + R_k\succeq m_B I$, so $\|B_k^{-1}\|\le \dfrac{1}{m_B}$.
Then
\[
y_k=-B_k^{-1}Z_k^\top V_k^\top g_k  
\implies 
\|y_k\|\le \dfrac{1}{m_B}\|Z_k^\top V_k^\top g_k\|
= \dfrac{1}{m_B}\|V_k^\top g_k\|
\le \dfrac{1}{m_B}\|g_k\|.
\]
Since $V_k$ has orthonormal columns and $Z_k$ is orthogonal, $\|s_k\|=\|V_k Z_k y_k\|=\|y_k\|$.\\
(2) Let $R_k^{exact}:=(H_k+S_k)s_k + g_k = R_k^{(N)} +S_ks_k$ where $R_k^{(N)} = H_ks_k+ g_k$. 

Consider $R_k^{(N)}=(I-P_k)R_k^{(N)} + P_k R_k^{(N)}$ and $s_k = P_ks_k$, we have
\begin{equation}\label{eq:1}
\begin{aligned}
\|(I-P_k)R_k^{(N)}\| 
&= \|(I-P_k)H_kP_k s_k + (I-P_k)g_k \| \\[4pt]
&\le\|(I-P_k)H_kP_k\|\|s_k\| +\|(I-P_k)g_k\|\\[4pt]
&= \|H_kP_k - P_kH_kP_k\|\,\|s_k\| +\|(I-P_k)g_k\|.
\end{aligned}
\end{equation}
By multiplying the reduced regularized system by $V_kZ_k$, we also have
\[
\begin{aligned} 
V_kZ_kZ_k^\top V_k^\top H_k V_kZ_ky_k + V_kZ_kR_k y_k
&= -V_kZ_kZ_k^\top V_k^\top g_k \\[4pt]
P_kH_k s_k + V_kZ_kR_k y_k &= -P_k g_k \\[4pt]
P_k(H_k s_k + g_k) &= -V_kZ_kR_k y_k \\[4pt]
P_k R_k^{(N)} &= -V_kZ_kR_k y_k.
\end{aligned}
\]
Then 
\begin{equation}\label{eq:2}
\|P_k R_k^{(N)}\| = \|-V_kZ_kR_k y_k\|
\le \|V_k\|\|Z_k\|\|R_k\|\|y_k\|
=  \|R_k\|\|y_k\|
\end{equation}
Combining the estimates \ref{eq:1} and \ref{eq:2} yields the central local error recursion
\[
\begin{aligned}
  \|R_k^{exact}\|  
&\le\|R_k^{(N)}\| +\|S_ks_k\| \\[4pt]
&\le \|(I-P_k)R_k^{(N)}\| +\|P_k R_k^{(N)}\| +\|S_ks_k\|\ \\[4pt]
&\le\|H_kP_k - P_kH_kP_k\|\,\|s_k\|+\|(I-P_k)g_k\|+\|R_k\|\|y_k\|+\|S_k\|\|s_k\|.  
\end{aligned}
\]
\end{proof}
Using the exact Newton residual bound in Lemma ~\ref{lem:inexact}, we relate the local convergence rate to the combined subspace, curvature, regularization, and Hessian-approximation errors. Once full steps are accepted, a uniformly controlled residual yields local linear convergence, while vanishing residual components make the reduced LM step asymptotically equivalent to the exact Newton step and therefore produce superlinear convergence.
\begin{theorem}[Local linear and superlinear convergence]\label{thm:local}
Assume Assumptions~\ref{ass:A4}--\ref{ass:L3}.
Let $e_k:=x_k-x^\star$ and assume eventually $t_k=1$.
Then:
\begin{enumerate}
\item If $\theta_k:=C(\gamma_k+\varepsilon_k+\|R_k\| +\|S_k\|)\le \bar\theta<1$ for all sufficiently large $k$, then $\|e_{k+1}\|\le q\|e_k\|$ for some $q<1$ (linear convergence).
\item If $\theta_k\to 0$, then $\|e_{k+1}\|/\|e_k\|\to 0$ (superlinear convergence).
\end{enumerate}
\end{theorem}
\begin{proof}
Since $\nabla^2F(x^\star)\succ 0$ and $\nabla^2F$ is Lipschitz near $x^\star$, there exist $m_1,M_1>0$ and a neighborhood $U$ such that
$m_1 I \preceq \nabla^2F(x)\preceq M_1 I$ for all $x\in U$ \citep{DennisSchnabel1996,NocedalWright2006}. Since $\nabla F(x^\star) = 0$, we have 
\[
\begin{aligned}
g_k = \nabla F(x_k) - \nabla F(x^\star) 
&= \int_0^1 \nabla^2 F(x^\star+te_k)e_kdt \qquad\qquad \text{(By Fundamental Theorem of Calculus)} \\[4pt]
&= \int_0^1 \left[\nabla^2 F(x^\star+te_k) - \nabla^2 F(x_k)e_k\right] dt +\int_0^1 \nabla^2 F(x_k)e_kdt \\[4pt]
&= h_k+\nabla^2F(x_k)e_k.
\end{aligned}
\]
where $h_k = \int_0^1 \left[\nabla^2 F(x^\star+te_k) - \nabla^2 F(x_k)\right]e_k dt$. Using Hessian Lipschitz continuity, we also have 
\[
\|h_k\| \le \|e_k\|\int_0^1 \|\nabla^2 F(x^\star+te_k) - \nabla^2F(x_k)\| dt\le \|e_k\|\int_0^1 L_2(1-t)\|e_k\| dt = \dfrac{L_2}{2}\|e_k\|^2
\]
Define the Newton residual $R_k^{exact}:=\nabla^2F(x_k) s_k + g_k$.
Since $H(x)=J(x)^{\top}J(x)$ is continuous in a neighborhood of $x^\ast$, there exists $M_H>0$ such that $\|H_k\|\le M_H$ for all sufficiently large $k$. By Lemma~\ref{lem:inexact} and Assumption~\ref{ass:L2},
\[
\begin{aligned}
\|R_k^{exact}\| 
&\le \|H_kP_k - P_kH_kP_k\|\,\|s_k\|+\|(I-P_k)g_k\|+\|R_k\|\|y_k\|+\|S_k\|\|s_k\| \\[4pt]
&\le \epsilon_k\|H_k\|\dfrac{1}{m_B}\|g_k\|+\gamma_k\|g_k\|+\|R_k\|\dfrac{1}{m_B}\|g_k\|+\|S_k\|\dfrac{1}{m_B}\|g_k\| \\[4pt]
&\le (\dfrac{M_H\epsilon_k}{m_B}+\gamma_k+\dfrac{\|R_k\|}{m_B}+\dfrac{\|S_k\|}{m_B})\|g_k\|.
\end{aligned}
\]
At $t_k=1$, $e_{k+1}=e_k+s_k$ and
\[
\nabla^2F(x_k) e_{k+1} = \nabla^2F(x_k) e_k + \nabla^2F(x_k) s_k = (g_k-h_k) + (R_k^{exact}-g_k) = R_k^{exact}-h_k.
\]
Then,
\begin{equation}\label{eq:3}
\|e_{k+1}\|\le \|\nabla^2F(x_k)^{-1}\|(\|R_k^{exact}\|+\|h_k\|)
\le \dfrac{1}{m_1}\left[\left(\dfrac{M_H\epsilon_k}{m_B}+\gamma_k+\dfrac{\|R_k\|}{m_B}+\dfrac{\|S_k\|}{m_B}\right)\|g_k\|+\dfrac{L_2}{2}\|e_k\|^2\right].    
\end{equation}
We also have 
\begin{equation}\label{eq:4}
\|g_k\|=\|\nabla^2F(x_k) e_k + h_k\|\le \|\nabla^2F(x_k)\| \|e_k\| + \|h_k\| \le M_1\|e_k\|+\dfrac{L_2}{2}\|e_k\|^2
\end{equation}
Combine \ref{eq:3} and \ref{eq:4}, we have
\[
\|e_{k+1}\|\le a\|e_k\|^2 + \,\theta_k \|e_k\| \qquad \text{ for some constant } a>0
\]
If $\theta_k\le \bar\theta<1$ eventually, choose $\bar\theta<q<1$ and take sufficient large 
$k$ so that $a\|e_k\|\le q - \bar\theta$, yielding $\|e_{k+1}\|\le q\|e_k\|$.\\
If $\theta_k\to 0$, then $\|e_{k+1}\|=o(\|e_k\|)$, giving superlinear convergence.
\end{proof}
\begin{remark}
The local convergence analysis above is carried out using the Gauss--Newton/LM curvature approximation
$H(x)=J(x)^\top J(x)$. Therefore, the superlinear local behavior established in this section is only applicable to the zero-residual (or very small) regime, i.e.,
\[
r(x^\star)\approx0.
\]
In contrast, if $r(x^\star)\neq 0$, then the additional term $S(x)$ in $\nabla^2F(x^\star)$ is
generally nonzero, so $J(x^\star)^\top J(x^\star)$ does not coincide with the true Hessian.
In that case, the HSLM method is typically expected to be at best locally linear.
\end{remark} 
\section{Numerical Experiments}
We evaluate the HSLM algorithm by training feedforward neural networks for nonlinear function approximation. Feedforward neural networks are suitable test problems because their training can be formulated as a nonlinear least-squares problem whose size can be systematically increased by varying the network architecture and training dataset, while preserving the same approximation objective. Their use is motivated by their ability to represent complex nonlinear relationships and by classical universal approximation results \citep{cybenko1989approximation, hornik1989multilayer,barron1993universal}. Although the classical LM method is known for rapid convergence and strong robustness when training small- to medium-sized neural networks \citep{Ampazis,KimLee,yu2018levenberg}, its computational and memory costs increase significantly with the number of parameters and data points. In this section, we compare classical LM, Krylov-subspace LM, and HSLM across different network and dataset sizes and measure performance using training and validation errors, convergence behavior, and computational cost. Our experiments evaluate whether HSLM can maintain convergence behavior and solution quality comparable to classical LM while reducing computational cost as the problem size increases.

\subsection{Experimental Setup}
The neural network is trained to approximate the following Friedman function:
\[
f(\mathbf{z}) = 10\sin(\pi z_1z_2)+20(z_3-\dfrac{1}{2})^2 +10z_4+5z_5
\]
where $\mathbf{z} = (z_1,z_2,...,z_7)\in [-1,1]^{7}$. This modified nonlinear-regression benchmark, which was introduced by Friedman \cite{friedman1991mars}, combines trigonometric, linear, and polynomial components, making it well-suited for evaluating the nonlinear least-squares algorithms. 

Training input samples are drawn uniformly from $[-1,1]^{7}$, and corresponding outputs are computed from Friedman's function. Three training-set sizes are considered (Table~\ref{tab:network_configs}), while 1000 noise-free samples are reserved for validation. To assess robustness, additive zero-mean Gaussian noise with standard deviation equal to 5\% of the standard deviation of the target signal, $\sigma_Y$, is applied to the training output,
\[
\mathbf{\hat{y}} = f(\mathbf{z})+\epsilon \qquad\text{where}\quad \epsilon \sim \mathcal{N}\\(\mathbf{0},\sigma^2),\quad \sigma =0.05\sigma_Y
\]

A fully connected multilayer perceptron (MLP) with two hidden layers is used to approximate the target function \cite{yu2018levenberg}. For an input vector $\mathbf{z}\in \mathbf{R}^{7}$, the hidden layers use hyperbolic tangent activation function, while the output layer is linear
\[
\begin{aligned}
\mathbf{h}^{(1)}
    &= \tanh\!\left(\mathbf{W}^{(1)}\mathbf{z}
       + \mathbf{b}^{(1)}\right), \\
\mathbf{h}^{(2)}
    &= \tanh\!\left(\mathbf{W}^{(2)}\mathbf{h}^{(1)}
       + \mathbf{b}^{(2)}\right), \\
\mathbf{y}
    &= \mathbf{W}^{(3)}\mathbf{h}^{(2)}
       + \mathbf{b}^{(3)}.
\end{aligned}
\]
where $\mathbf{W}^{(l)}$ and $\mathbf{b}^{(l)}$ denote the weight matrices and bias vectors of layer $l$, respectively. 

Three network architectures, with increasing numbers of training parameters ($\mathbf{W}^{(l)}$ and $\mathbf{b}^{(l)}$) are described in Table~\ref{tab:network_configs}, allowing the dimension of the nonlinear least-squares problem to systematically increase while using the same underlying regression task.
\begin{table}[H]
\centering
\caption{Neural network configurations with increasing model complexity and training data size. Network architecture is reported as the number of neurons in the input layer, the two hidden layers, and the output layer, respectively.}
\begin{tabular}{c c c c}
\hline
\textbf{Network} & \textbf{Architecture} & \textbf{Trainable Parameters}  & \textbf{Training data size}\\
\hline
1 & $7 \rightarrow 35 \rightarrow  20 \rightarrow 1$ & 1021 & 10000\\
2 & $7 \rightarrow 60  \rightarrow 25 \rightarrow 1$ & 2031 & 20000\\
3 & $7 \rightarrow 80  \rightarrow 40 \rightarrow 1$ & 3921 & 40000\\
\hline
\end{tabular}
\label{tab:network_configs}
\end{table}

By collecting all trainable weights and biases into a single parameter vector $\mathbf{x}$, training the neural network can be formulated as a nonlinear least-squares problem
\[
\min_\mathbf{x}\dfrac{1}{2}\|\mathbf{r(x)}\|^2
\]
where the residual vector is 
\[
\mathbf{r}(\mathbf{x})
=
\begin{bmatrix}
y_1(\mathbf{x})-\hat{y}_1 \\
y_2(\mathbf{x})-\hat{y}_2 \\
\vdots \\
y_N(\mathbf{x})-\hat{y}_N
\end{bmatrix},
\]
with $y_i(\mathbf{\mathbf{x}})$, $\hat{y_i}$ denoting the network prediction for the $i^{th}$ training input and the corresponding noisy target output. The Jacobian matrix 
\[
\mathbf{J}(\mathbf{x}) = \dfrac{\partial \mathbf{r}}{\partial \mathbf{x}}
\]
is computed efficiently using a modified backpropagation procedure, which is described in detail in \cite{HaganMenhaj1994}. 

In our implementation, the maximum subspace dimension in both Krylov subspace LM and HSLM is limited to 10\% of the parameter dimension. HSLM constructs the reduced basis using the hybrid strategy described in Section 3, whereas Krylov-subspace LM uses Lanczos-based projection of the scaled LM system, consistent with the Krylov subspace framework described in \cite{Lin2016}. Additional implementation details are provided in Appendix A.
\subsection{Results}
We compare the three algorithms in terms of convergence behavior and computational efficiency as the size of the nonlinear least-squares problem increases. For each network architecture, 50 trials are performed using randomly generated initial weights $\mathbf{W^{(l)}}$ and bias $\mathbf{b^{(l)}}$. Figure 1 shows five representative runs for the largest network using identical initializations across the three algorithms, while Table~\ref{tab:nn_training_comparison1} summarizes the average number of iterations and computational time per iteration over all 50 trials. We use the mean squared error (MSE) to quantify both the training error and the validation error.

\begin{figure}[H]
\centering\includegraphics[width=1\linewidth]{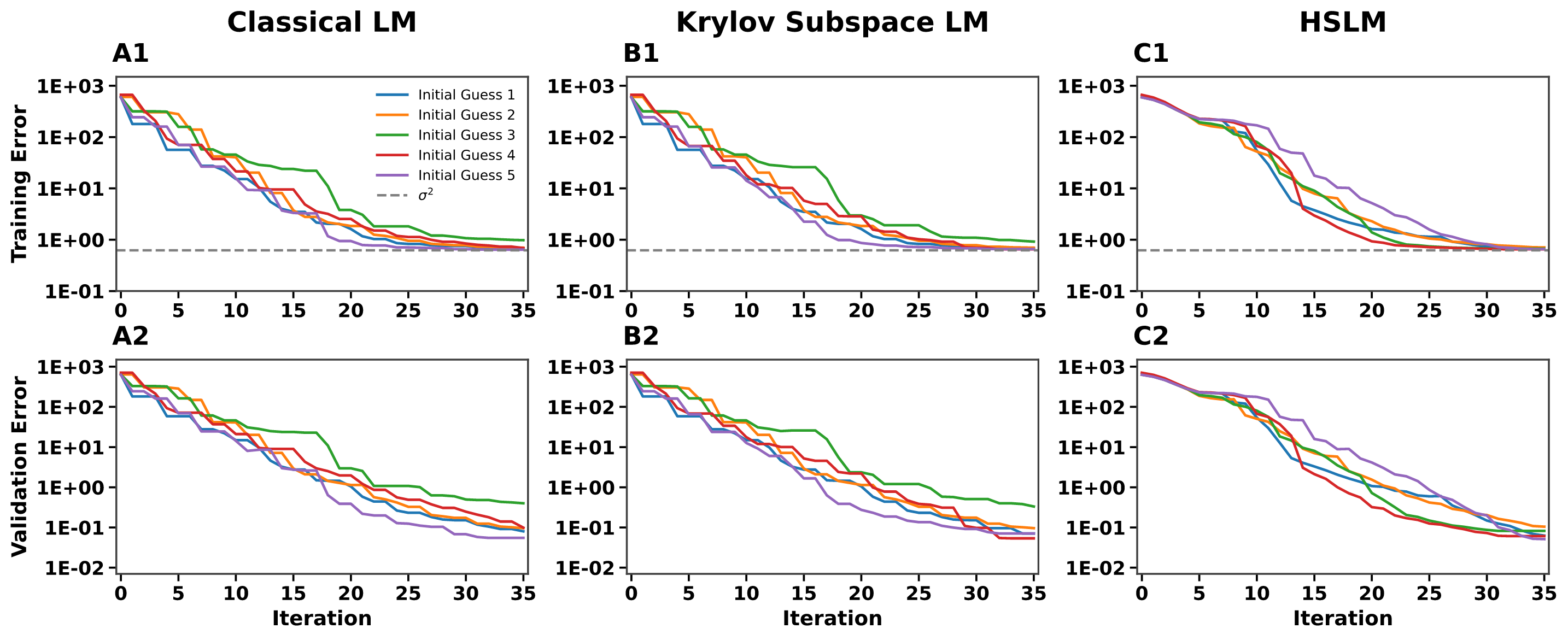}
\caption{Training results for five representative runs of network 3 for nonlinear function approximation. A1-A2. Levenberg–Marquardt algorithm. B1-B2. Krylov subspace Levenberg–Marquardt algorithm. C1-C2. Hybrid subspace Levenberg–Marquardt algorithm. The gray dashed line indicates the variance of the added noise, $\sigma^2 \approx 0.622$. These results indicate that HSLM achieves convergence behavior comparable to the classical LM and Krylov-subspace LM algorithms.}
\label{fig:subim1}
\end{figure} 
\begin{table}[H]
\centering
\caption{Comparison in Computational Performance of Neural Network Training Algorithms}
\label{tab:nn_training_comparison1}
\renewcommand{\arraystretch}{1.15}
\begin{tabular}{lcccccc}
\toprule
& \multicolumn{2}{c}{Classical LM} & \multicolumn{2}{c}{Krylov Subspace LM}& \multicolumn{2}{c}{HSLM} \\
\cmidrule(lr){2-3}\cmidrule(lr){4-5}\cmidrule(lr){6-7}

& Iteration & Time / Iteration (s)
&  Iteration & Time / Iteration (s) & Iteration & Time / Iteration (s)\\
\midrule
Network 1   & 41.7 & 0.492 &  41.1 & 0.262 &  52.7 & 0.131\\
Network 2   & 41.6  & 2.367  & 39.4  & 1.738  &  36.9 & 0.615\\
Network 3   & 34.9   & 15.830    & 36.5   & 13.082   &  35.4 & 3.665\\
\bottomrule
\end{tabular}
\end{table}

All three algorithms achieve similar final training errors, with the training error approaching the variance of the added noise, $\sigma^2$, indicating that the networks capture the underlying Friedman function to approximately the noise level (see Appendix B for network 1 \& 2). Despite this comparable convergence behavior, their computational costs differ substantially. HSLM reduces the average execution time per iteration relative to classical LM by factors of approximately 3.7, 3.8, and 4.3 for Networks 1, 2, and 3, respectively. Relative to Krylov-subspace LM, the corresponding speedups are approximately 2.0, 2.8, and 3.6. The increasing advantage with network size suggests that the reduced adaptive formulation becomes increasingly beneficial as the cost of the full LM system grows.

HSLM achieves this computational efficiency using relatively small adaptive subspaces. The deterministic adequacy monitor adjusts the subspace dimension according to the information required at each iteration. As shown in Appendix C, HSLM generally operates with a substantially smaller subspace than Krylov-subspace LM and still satisfies the prescribed projected-gradient adequacy criterion. The convergence behavior also depends on the composition of the subspace. Appendix D shows that randomized curvature probes alone, or randomized probes combined only with the gradient, converge more slowly than the full hybrid construction. This comparison supports the use of complementary gradient, memory, Krylov/Lanczos, and randomized curvature information in HSLM.

\section{Discussion}

We have developed an adaptive hybrid subspace Levenberg–Marquardt method for large-scale nonlinear least-squares problems. The framework provides convergence guarantees for a general hybrid subspace construction, so the choice of basis directions can be tailored to the structure of the problem. The adequacy monitor provides a systematic mechanism for integrating complementary basis directions while adaptively controlling the subspace dimension. The numerical experiments illustrate both aspects of this construction: Appendix C shows how the subspace dimension adapts during optimization, while Appendix D demonstrates the benefit of combining multiple sources of descent and curvature information. Because the convergence framework does not prescribe a specific candidate pool, the subspace construction can incorporate problem-dependent directions, providing flexibility for other classes of large-scale nonlinear least-squares problems \cite{Martens2010,Vinyals2012,Lin2016,Bjarkason2018}.

In addition, separating step acceptance from damping adjustment allows HSLM to retain adaptive regularization while decoupling the two mechanisms \citep{Armijo1966,NocedalWright2006}. Armijo backtracking determines whether and how far to move along the computed direction, whereas the gain ratio is used only to adjust the damping parameter based on the quality of the local quadratic approximation \citep{ConnGouldToint2000,MadsenNielsenTingleff2004,MoreSorensen1983}. This separation avoids the cost of repeated damped-system solves during step acceptance without eliminating the stabilizing role of LM regularization. Moreover, the predicted reduction remains positive for accepted Armijo step lengths and can be evaluated directly in the reduced coordinates (Appendix F), ensuring that the gain ratio is well defined.

The adequacy monitor ensures that the reduced space captures a prescribed fraction of the gradient, but it does not directly measure whether the subspace captures all curvature directions relevant to an effective LM step. Performance therefore still depends on the quality and diversity of the candidate directions used to construct and enrich the subspace. The hybrid construction used here addresses this issue by combining gradient, memory, Krylov/Lanczos, and randomized curvature information, but other problems may benefit from different or problem-specific candidate directions.

The scaling experiments in Appendix E help clarify a separate limitation related to computational cost. At a fixed training-set size, HSLM becomes increasingly advantageous as the parameter dimension grows, indicating that the reduced-space construction is particularly effective in reducing parameter-dependent linear-algebra costs. However, these experiments do not examine the effect of increasing the number of training samples. As the training-set size grows, the cost of Jacobian evaluations and matrix-vector products is expected to become increasingly important and is not reduced directly by restricting the optimization step to a low-dimensional parameter subspace.  In addition, the present numerical study is limited to neural-network regression problems. Future work should therefore examine mini-batch and matrix-free HSLM variants for larger training sets, as well as applications to large-scale inverse problems with different Jacobian and curvature structures.

Overall, the results demonstrate that adaptive hybrid subspace construction can preserve convergence behavior comparable to full-space LM while substantially reducing the cost of linear-algebra computations as the parameter dimension increases. Together with the convergence analysis, these results suggest that HSLM provides a promising approach for retaining the favorable properties of LM-type methods while improving computational efficiency for large-scale problems.

\newpage\section*{Appendix}
\appendix
\section{Implementation Details for the Subspace 
Levenberg--Marquardt Methods}
\label{app:subspace_lm}

We implemented two subspace variants of the Levenberg--Marquardt method: a Krylov-subspace LM method and HSLM method. For the updating damping strategy, we use delayed gratification for all algorithms by choosing $\rho_{low} = \rho_{high} = 0$ and $\alpha = 10^{-3}$. This appendix summarizes the main subspace construction rules, numerical tolerances, and expansion criteria for each method.

\subsection{Krylov-Subspace Levenberg--Marquardt Method}
\label{app:krylov_lm}

The Krylov basis is generated by the Lanczos process, initialized with the normalized gradient $q_1=g_k/\|g_k\|_2$, which spans the Krylov subspace
\[
\mathcal{K}_{p}(H_k,g_k)
=
\operatorname{span}
\{g_k,H_kg_k,\ldots,H_k^{p-1}g_k\}.
\]
To reduce computational and memory costs, we avoid explicitly forming the Gauss–Newton matrix, $H_kv=J_k^\top(J_kv)$, and instead compute the required matrix–vector products directly. The Lanczos iteration is terminated when either the Krylov subspace reaches the prescribed maximum dimension
$d_{\max}=\lfloor0.1n\rfloor$ or the Lanczos residual satisfies $\beta_j\le10^{-5}$.

Let $Q_k=[q_1,\ldots,q_{d_k}]$ denote the resulting orthonormal Krylov basis and
$T_k=Q_k^\top H_kQ_k$ the projected tridiagonal matrix. Since the first Lanczos vector is the normalized gradient,
\[
Q_k^\top g_k=\|g_k\|_2e_1,
\]
where $e_1=[1,0,\ldots,0]^\top\in\mathbb{R}^{d_k}$ is the first canonical basis vector. Consequently, the reduced Levenberg--Marquardt system is
\[
(T_k+\mu_k \mathbf{I})z_k=-\|g_k\|_2e_1,
\]
and the full-space step is recovered as $s_k=Q_kz_k.$ If a trial step is rejected, the previously computed Krylov basis is retained and only the reduced system is solved with the updated damping parameter. This avoids recomputing the Jacobian and repeating the Lanczos process, thereby reducing the computational cost of rejected iterations.

\subsection{Adaptive Hybrid Reduced-Subspace 
Levenberg--Marquardt Method}
\label{app:hybrid_lm}

At iteration~$k$, the hybrid method constructs a reduced orthonormal basis $V_k\in\mathbb{R}^{n\times p}$ by combining Lanczos--Krylov directions with stochastic curvature probes. The Lanczos process is initialized with the normalized gradient $q_1=g_k/\|g_k\|_2$, and uses implicit Gauss--Newton matrix--vector products, $H_kv=J_k^\top(J_kv)$. The initial candidate subspace consists of $\lfloor0.01n\rfloor$ random curvature probes and recent accepted step $s_{k-1}$. Each random probe is generated as
\[
d_j=J_k^\top J_k\omega_j,
\qquad
\omega_j\sim\mathcal{N}(0,I),
\]
normalized, and appended to the candidate matrix. The combined candidate matrix is then orthogonalized using a thin QR factorization, with linearly dependent directions removed. The quality of the reduced basis is measured by the fraction of the squared gradient norm captured by the subspace. If this quantity satisfies $\eta^{\mathrm{sub}}_k\ge0.99$, the current basis is accepted. Otherwise, the subspace is expanded by adding $\lfloor0.02n\rfloor$ Lanczos directions and $\lfloor0.01n\rfloor$ random curvature probes. 

Subspace expansion continues until either the target $\eta^{\mathrm{sub}}_k\ge0.99$ is achieved or the maximum dimension reaches $\lfloor0.1n\rfloor$. Since linearly dependent vectors are discarded during QR orthogonalization, the final numerical dimension of $V_k$ may be smaller than the total number of generated candidate directions.

The hybrid method also employs backtracking, starting with $t_k=1$ and updating $t_k\leftarrow \beta t_k$, with at most ten trial step lengths evaluated, which is an implementation safeguard outside the theorem. The algorithm is terminated when the reduction in the objective function satisfies $F(x_k)-F(x_{k+1})<0.005,$ or when the maximum number of $10000$ iterations is reached.

\subsection{Summary of Numerical Settings}

\begin{table}[htbp]
    \centering
    \caption{Numerical settings used for the Krylov and adaptive
    hybrid reduced-subspace LM methods.}
    \label{tab:subspace_lm_settings}
    \small
    \begin{tabular}{lcc}
        \toprule
        Setting
        & Krylov LM
        & Hybrid reduced-subspace LM \\
        \midrule

        Initial Krylov directions
        & Up to maximum
        & --- \\

        Initial random probes
        & ---
        & $0.01n$ \\

        Krylov directions per expansion
        & ---
        & $0.02n$ \\

        Random probes per expansion
        & ---
        & $0.01n$ \\

        Adequacy threshold
        & ---
        & $\eta_{\min}=0.99$ \\

        Lanczos tolerance
        & $10^{-5}$
        & $10^{-5}$ \\

        QR dependence tolerance
        & ---
        & $10^{-12}$ \\

        Reduced singular-value floor
        & ---
        & $10^{-8}$ \\

        Initial damping $\mu_0$
        & $10$
        & $10$ \\

        Damping decrease $\mu_{down}$
        & 2
        & 2 \\

        Damping increase
        & 5
        & 5 \\

        Step stopping tolerance
        & $0.005$
        & $0.005$ \\

        Maximum iterations
        & $10000$
        & $10000$ \\
        \bottomrule
    \end{tabular}
\end{table}

\newpage
\section{Training and validation results for network 1 \& 2}
Five representative training results obtained from Network 1 \& 2 using the same initial guess demonstrate comparable convergence behavior for the three algorithms. A1-A2. Levenberg–Marquardt algorithm. B1-B2. Krylov subspace Levenberg–Marquardt algorithm. C1-C2. Hybrid subspace Levenberg–Marquardt algorithm. The gray dashed line indicates the variance of the added noise, $\sigma^2 \approx 0.622$. 
\subsection{Network 1}
\begin{figure}[H]
\centering\includegraphics[width=1\linewidth]{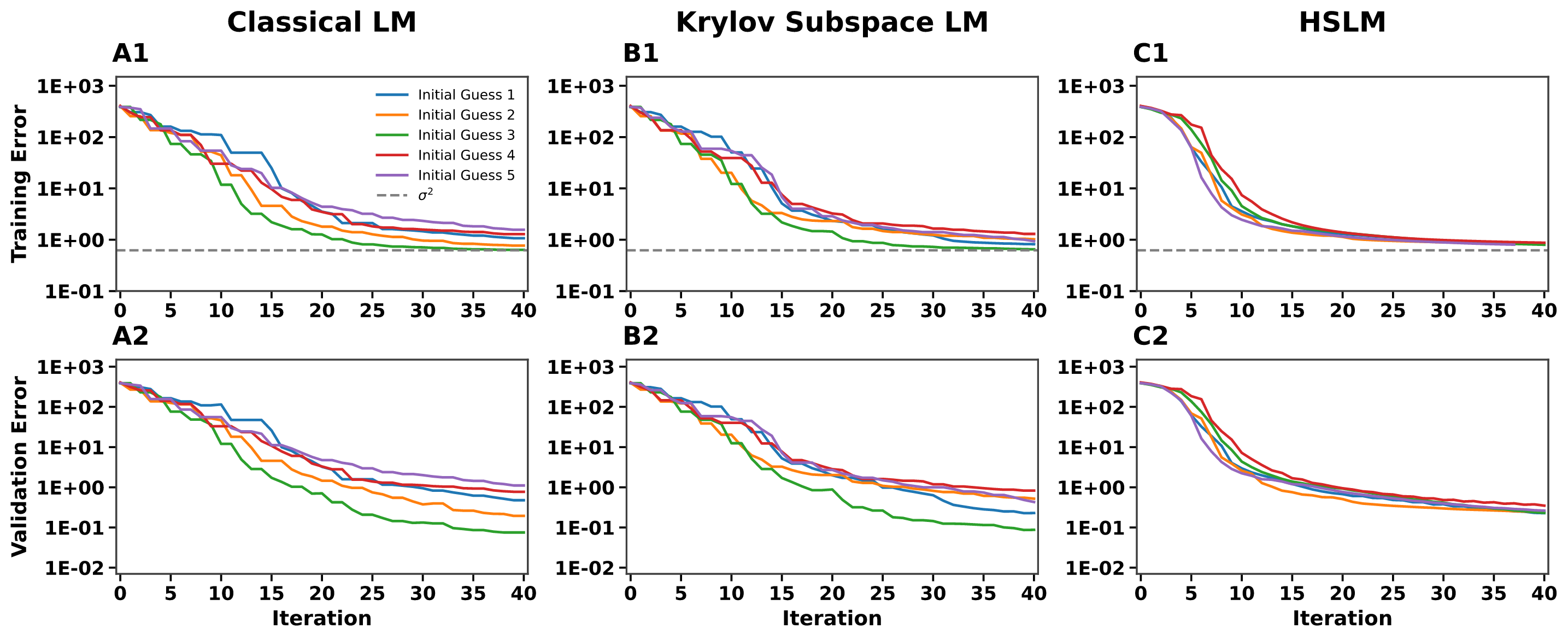}
\end{figure} 
\subsection{Network 2}
\begin{figure}[H]
\centering\includegraphics[width=1\linewidth]{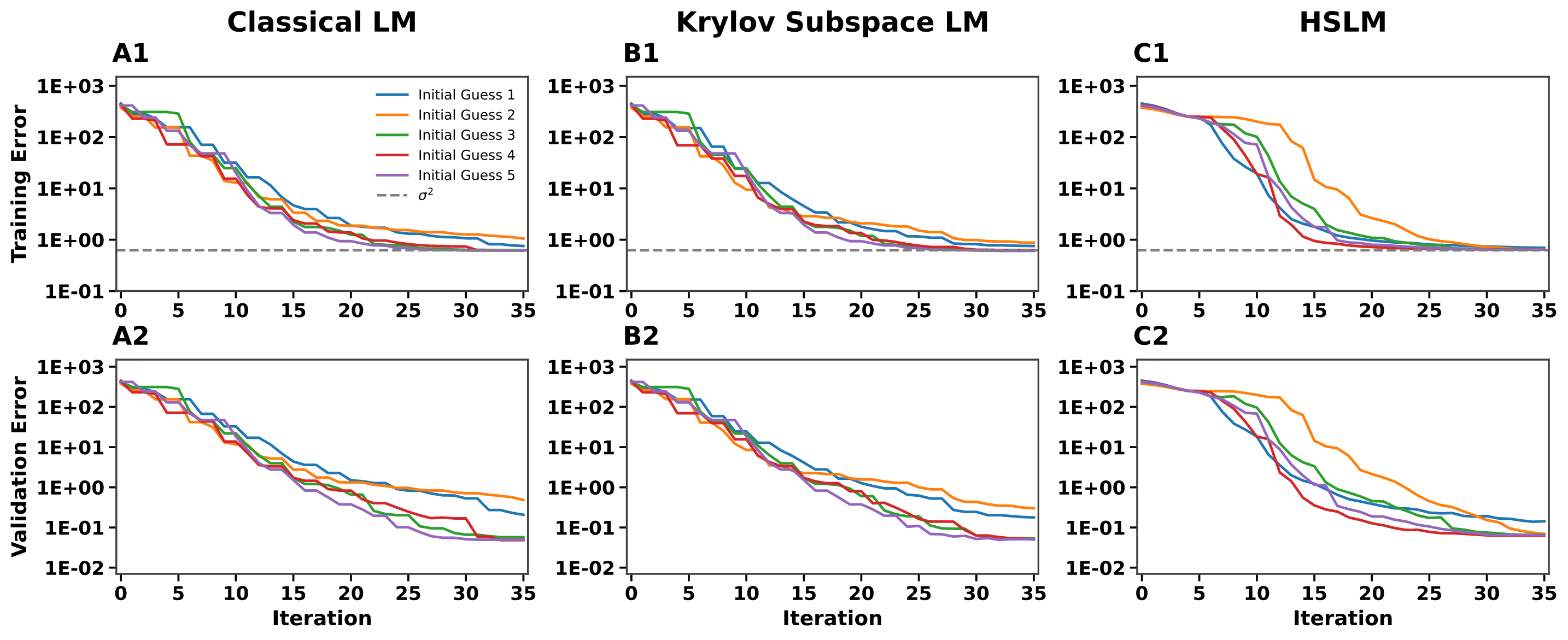}
\end{figure}

\newpage
\section{Adaptive subspace dimension selection using deterministic adequacy monitor}
\begin{figure}[H]
\centering\includegraphics[width=1\linewidth]{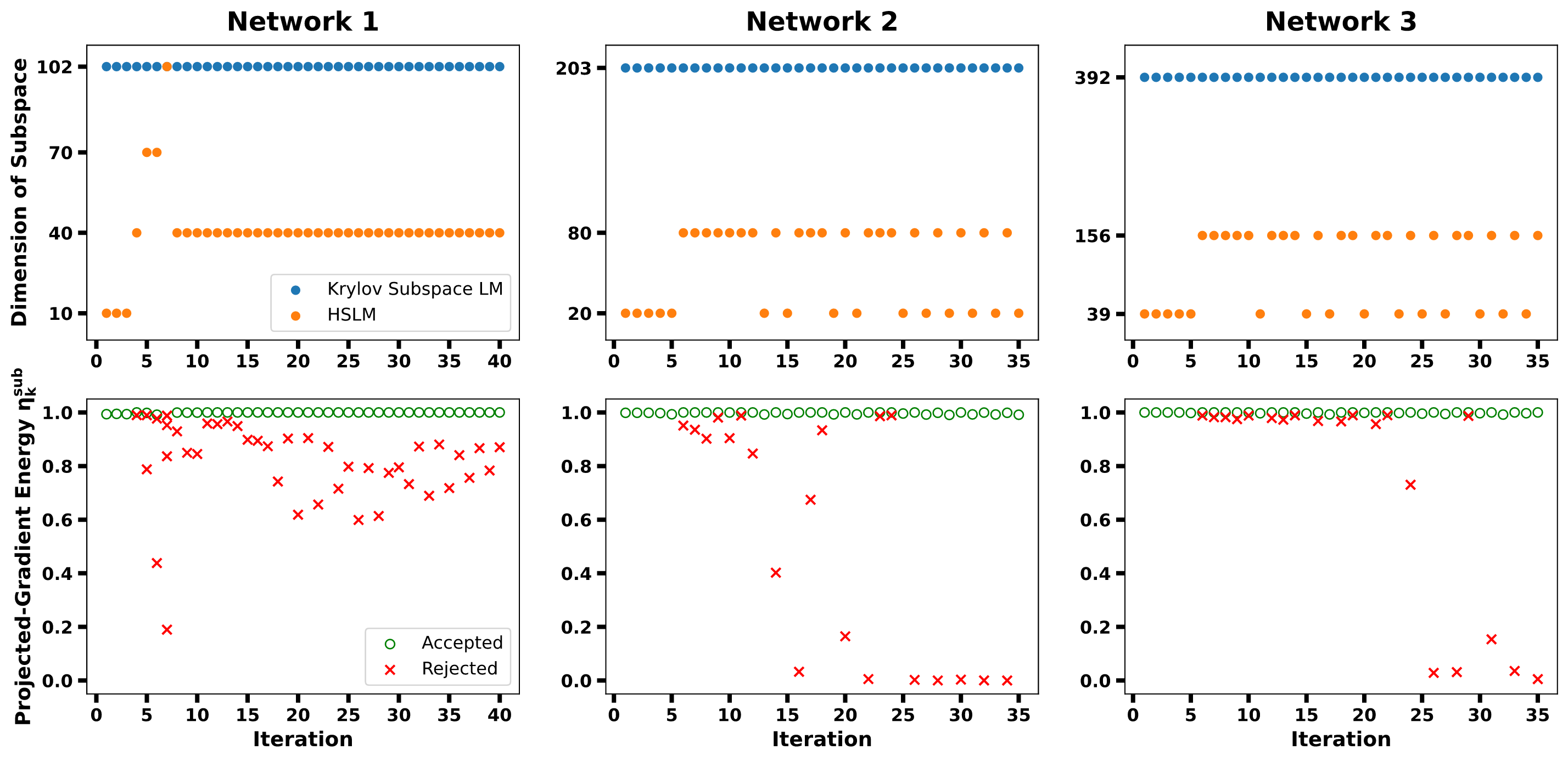}
\end{figure}
The figure compares the subspace dimensions selected by Krylov Subspace LM and HSLM across iterations for three neural networks of increasing complexity. In the upper row, each panel shows a representative optimization run using the same initial guess for the corresponding network. In the lower row, each panel displays the corresponding projected-gradient energy, $\bf{\eta^{\mathrm{sub}}_k}$, generated by HSLM from the same initial guess, illustrating how failure of the adequacy criterion triggers subspace enrichment during optimization. In contrast to Krylov-subspace LM, which uses a fixed subspace equal to 10\% of the parameter dimension, HSLM dynamically adjusts the subspace dimension and generally satisfies the adequacy criterion using a substantially smaller subspace.

\newpage
\section{Influence of Candidate Basis Sources in Subspace Construction on Algorithm Performance}
To examine how different sources of basis candidates in subspace construction influence the performance of the proposed method, we use the experimental setup described above for Network 3 to compare three difference subspace construction strategies: randomized Hessian probes only, randomized Hessian probes and gradient direction, and HSLM.

\begin{figure}[H]
\centering\includegraphics[width=1\linewidth]{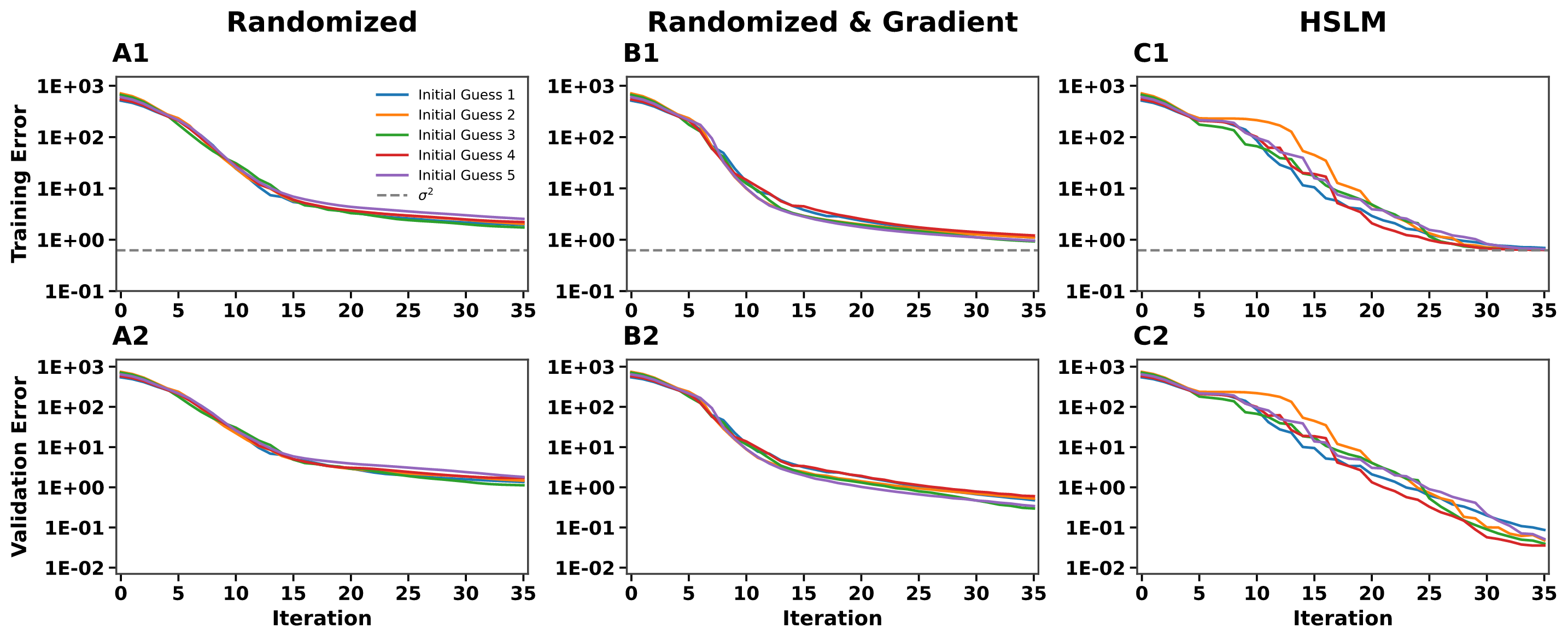}
\end{figure} 
Five representative training results obtained from Network 3 using the same initial guess demonstrate comparable convergence behavior for the three algorithms. A1-A2. Randomized Hessian/Gauss--Newton matrix--vector only B1-B2. Randomized Hessian/Gauss--Newton matrix--vector and gradient direction. C1-C2. HSLM algorithm. The gray dashed line indicates the variance of the added noise, $\sigma^2 \approx 0.622$. 
The figure compares training and validation errors for Network 3 using five representative initial guesses.  The purely randomized Hessian method, whose subspace is constructed only from randomized Hessian probes, initially reduces error but struggles to reach the added noise level, indicating that the sampled curvature directions alone may omit important descent information. Adding the gradient direction produces more reliable convergence behavior toward a local optimal solution, but the reduction in both training and validation error remains slow. Meanwhile, HSLM achieves the strongest overall performance, reaching the lowest training and validation errors more rapidly and consistently across the five representative runs. These results demonstrate that augmenting randomized curvature information with additional structured directions improves the convergence behavior, with the full HSLM construction providing the strongest performance among the three variants considered.

\newpage
\section{Computational Efficiency of HSLM in Increasing Parameter Dimension at Fixed Training-Set Size}
We evaluate the efficiency of the proposed HSLM algorithm by comparing it with the hybrid Krylov-subspace and classical Levenberg--Marquardt algorithms on overparameterized feedforward neural networks. The experimental setup is identical to that described previously, using the same 40,000 training samples. To simulate measurement uncertainty, additive Gaussian noise with a standard deviation equal to 5\% of the standard deviation of the target signal is added to the training data. The average number of iterations and the average execution time per iteration over 50 independent trials are summarized in Table~\ref{tab:nn_training_comparison2}.
\begin{table}[H]
\centering
\caption{Comparison between algorithms for neural-network training}
\label{tab:nn_training_comparison2}
\renewcommand{\arraystretch}{1.15}
\begin{tabular}{lcccccc}
\toprule
& \multicolumn{2}{c}{Classical LM} & \multicolumn{2}{c}{Krylov Subspace LM}& \multicolumn{2}{c}{HSLM} \\
\cmidrule(lr){2-3}\cmidrule(lr){4-5}\cmidrule(lr){6-7}

& Iteration & Time / Iteration (s)
&  Iteration & Time / Iteration (s) & Iteration & Time / Iteration (s)\\
\midrule
Network 1   & 43.9 & 1.044 &  45.3 & 0.941 &  49.8 & 0.464\\
Network 2   & 38.2  & 3.235  & 37.1  & 3.790  &  37.1 & 1.221\\
Network 3   & 34.9   & 15.830    & 36.5   & 13.082   &  35.4 & 3.665\\
\bottomrule
\end{tabular}
\end{table}
For all three network architectures, the final training error is approximately equal to the variance of the additive noise, indicating that each algorithm successfully captures the underlying Friedman function and that the remaining error is primarily due to irreducible noise. While all three methods exhibit comparable convergence behavior, the proposed HSLM algorithm scales more effectively as the network size increases. Compared with the classical LM algorithm, HSLM reduces the execution time by factors of approximately 2.3, 2.6, and 4.3 for Networks~1, 2, and~3, respectively. Relative to the Krylov-subspace LM algorithm, HSLM achieves additional speedups of approximately 2.0, 3.1, and 3.6 for Networks~1, 2, and~3, respectively, demonstrating its enhanced computational performance for large overparameterized networks.

\newpage
\section{Closed-form predicted reduction}
\begin{lemma}\label{lem:pred}
Let $m_k$ be defined by \eqref{eq:model} and $s_k$ by \eqref{eq:reducedLM}. Then for any $t\in(0,2)$,
\begin{equation}\label{eq:pred}
m_k(0)-m_k(t s_k)=\dfrac{1}{2}\, t(2-t)\, y_k^\top B_k y_k >0 \qquad \text{ for all } t\in(0,2)
\end{equation}

\end{lemma}
\begin{proof}
Using orthogonality of $V_k$ and $Z_k$, the local quadratic approximation \eqref{eq:model} restricted to this subspace becomes
\[
\begin{aligned}
m_k(t s_k)
&= \dfrac{1}{2} \|r_k + t J_k s_k\|^2 + \frac{\mu_k}{2} t^2 \|D_{k,p}^{1/2}Z_k^\top V_k^\top s_k\|^2 \\[4pt]
&= \dfrac{1}{2}\|r_k + t J_k(V_kZ_ky_k)\|^2 + \frac{\mu_k}{2} t^2 \|D_{k,p}^{1/2} Z_k^\top V_k^\top (V_kZ_ky_k)\|^2  \\[4pt]
&= \dfrac{1}{2}\|r_k + t (J_kV_k)Z_ky_k\|^2 + \frac{\mu_k}{2} t^2 \|D_{k,p}^{1/2} y_k\|^2  \\[4pt]
&= \dfrac{1}{2}\|r_k + t J_{k,p} Z_k y_k\|^2 + \frac{\mu_k}{2} t^2 y_k^\top D_{k,p} y_k .
\end{aligned}
\]
Since $J_{k,p}Z_k=U_k\Sigma_{k,p}$, define $q_k:=U_k^\top r_k$. Then
\[
\|r_k + t U_k\Sigma_{k,p} y_k\|^2
= \|r_k\|^2 + 2t\, y_k^\top \Sigma_{k,p} q_k + t^2 y_k^\top \Sigma_{k,p}^2 y_k,
\]
using orthogonality of $U_k$. Hence
\[
m_k(t s_k)
= \dfrac{1}{2} \|r_k\|^2 + t\, y_k^\top \Sigma_{k,p} q_k + \dfrac{1}{2} t^2 y_k^\top(\Sigma_{k,p}^2+\mu_k D_{k,p})y_k.
\]
Using the reduced LM system \eqref{eq:reducedLM}, $\Sigma_{k,p} q_k = -B_k y_k$, then
\[
m_k(t s_k)
= \dfrac{1}{2} \|r_k\|^2 - t\, y_k^\top B_k y_k + \dfrac{1}{2} t^2 y_k^\top B_k y_k
= m_k(0) + \dfrac{1}{2} t(t-2) y_k^\top B_k y_k,
\]
which rearranges to \eqref{eq:pred}. Then,

\[
m_k(0) - m_k(t s_k) 
= \dfrac{1}{2} t(2-t) y_k^\top B_k y_k > 0 \quad \text{ under Assumption 3}
\]
since $y_k^\top B_k y_k>0$ (Lemma~\ref{lem:spd}) and $t\in(0,2)$.
\end{proof}

\newpage
\bibliography{references}

@incollection{VogelWade1993,
  author    = {Vogel, C. R. and Wade, J. G.},
  title     = {A Modified Levenberg--Marquardt Algorithm for Large-Scale Inverse Problems},
  booktitle = {Computation and Control III},
  editor    = {Bowers, Kenneth and Lund, John},
  series    = {Progress in Systems and Control Theory},
  volume    = {15},
  publisher = {Birkh{\"a}user},
  address   = {Boston, MA},
  pages     = {367--378},
  year      = {1993},
  doi       = {10.1007/978-1-4612-0321-6_27}
}

@article{halko2011finding,
  title   = {Finding Structure with Randomness:
             Probabilistic Algorithms for Constructing
             Approximate Matrix Decompositions},
  author  = {Halko, Nathan and Martinsson, Per-Gunnar
             and Tropp, Joel A.},
  journal = {SIAM Review},
  volume  = {53},
  number  = {2},
  pages   = {217--288},
  year    = {2011},
  doi     = {10.1137/090771806}
}

@inproceedings{huang2020span,
  author    = {Huang, Xunpeng and Liang, Xianfeng and Liu, Zhengyang
               and Li, Lei and Yu, Yue and Li, Yitan},
  title     = {{SPAN}: A Stochastic Projected Approximate Newton Method},
  booktitle = {Proceedings of the AAAI Conference on Artificial Intelligence},
  pages     = {1520--1527},
  year      = {2020},
  doi       = {10.1609/aaai.v34i02.5511}
}

@article{HaganMenhaj1994,
  author  = {Hagan, Martin T. and Menhaj, Mohammad B.},
  title   = {Training Feedforward Networks with the Marquardt Algorithm},
  journal = {IEEE Trans. Neural Netw.},
  volume  = {5},
  number  = {6},
  pages   = {989--993},
  year    = {1994},
  doi     = {10.1109/72.329697}
}

@article{Bjarkason2018,
  author  = {Bjarkason, Elvar K. and Maclaren, Oliver J. and O'Sullivan, John P. and O'Sullivan, Michael J.},
  title   = {Randomized Truncated SVD Levenberg--Marquardt Approach to Geothermal Natural State and History Matching},
  journal = {Water Resources Research},
  volume  = {54},
  number  = {3},
  pages   = {2376--2404},
  year    = {2018},
  doi     = {10.1002/2017WR021870}
}

@article{Lin2016,
  author  = {Lin, Youzuo and O'Malley, Daniel and Vesselinov, Velimir V.},
  title   = {A Computationally Efficient Parallel {Levenberg--Marquardt} Algorithm for Highly Parameterized Inverse Model Analyses},
  journal = {Water Resour. Res.},
  volume  = {52},
  number  = {9},
  pages   = {6948--6977},
  year    = {2016},
  doi     = {10.1002/2016WR019028}
}

@article{Cartis2022,
  author  = {Cartis, Coralia and Fowkes, Jaroslav M. and Shao, Zhen},
  title   = {A Randomised Subspace Gauss-Newton Method for Nonlinear Least-Squares},
  journal = {CoRR},
  volume  = {abs/2211.05727},
  year    = {2022},
  doi     = {10.48550/arXiv.2211.05727},
  eprint  = {2211.05727},
  archivePrefix = {arXiv}
}

@article{Marquardt1963,
  author  = {Marquardt, Donald W.},
  title   = {An Algorithm for Least-Squares Estimation of Nonlinear Parameters},
  journal = {J. Soc. Indust. Appl. Math.},
  volume  = {11},
  number  = {2},
  pages   = {431--441},
  year    = {1963},
  doi     = {10.1137/0111030}
}

@phdthesis{Shao2021,
  author  = {Shao, Zhen},
  title   = {On Random Embeddings and Their Application to Optimisation},
  school  = {University of Oxford},
  address = {Oxford, UK},
  year    = {2021},
  doi     = {10.5287/ora-dmzo7zv6j}
}

@article{Levenberg1944,
  author  = {Levenberg, Kenneth},
  title   = {A Method for the Solution of Certain Non-Linear Problems in Least Squares},
  journal = {Quarterly of Applied Mathematics},
  volume  = {2},
  number  = {2},
  pages   = {164--168},
  year    = {1944}
}

@article{Armijo1966,
  author  = {Armijo, Larry},
  title   = {Minimization of Functions Having Lipschitz Continuous First Partial Derivatives},
  journal = {Pacific Journal of Mathematics},
  volume  = {16},
  number  = {1},
  pages   = {1--3},
  year    = {1966}
}

@book{NocedalWright2006,
   author    = {Nocedal, Jorge and Wright, Stephen J.},
  title     = {Numerical Optimization},
  edition   = {2nd},
  series    = {Springer Ser. Oper. Res. Financ. Eng.},
  publisher = {Springer},
  address   = {New York},
  year      = {2006},
  doi       = {10.1007/978-0-387-40065-5}
}

@book{ConnGouldToint2000,
  author    = {Conn, Andrew R. and Gould, Nicholas I. M. and Toint, Philippe L.},
  title     = {Trust Region Methods},
  series    = {MOS-SIAM Ser. Optim.},
  publisher = {SIAM},
  address   = {Philadelphia},
  year      = {2000},
  doi       = {10.1137/1.9780898719857}
}

@article{MoreSorensen1983,
  author  = {Mor{\'e}, Jorge J. and Sorensen, Danny C.},
  title   = {Computing a Trust Region Step},
  journal = {SIAM Journal on Scientific and Statistical Computing},
  volume  = {4},
  number  = {3},
  pages   = {553--572},
  year    = {1983}
}

@book{DennisSchnabel1996,
  author    = {Dennis, Jr., J. E. and Schnabel, Robert B.},
  title     = {Numerical Methods for Unconstrained Optimization and Nonlinear Equations},
  series    = {Classics in Applied Mathematics},
  volume    = {16},
  publisher = {SIAM},
  address   = {Philadelphia},
  year      = {1996},
  doi       = {10.1137/1.9781611971200}
}

@book{MadsenNielsenTingleff2004,
  author    = {Madsen, Kaj and Nielsen, Hans Bruun and Tingleff, Ole},
  title     = {Methods for Non-Linear Least Squares Problems},
  edition   = {2nd},
  publisher = {Informatics and Mathematical Modelling,
               Technical University of Denmark},
  address   = {Lyngby, Denmark},
  year      = {2004}
}

@inproceedings{Brand2002incremental,
  author    = {Brand, Matthew},
  title     = {Incremental Singular Value Decomposition of Uncertain Data with Missing Values},
  booktitle = {European Conference on Computer Vision (ECCV)},
  series    = {Lecture Notes in Computer Science},
  volume    = {2350},
  pages     = {707--720},
  year      = {2002},
  doi       = {10.1007/3-540-47969-4_47}
}

@article{yu2018levenberg,
  author  = {Wilamowski, Bogdan M. and Yu, Hao},
  title   = {Improved Computation for Levenberg--Marquardt Training},
  journal = {IEEE Transactions on Neural Networks},
  volume  = {21},
  number  = {6},
  pages   = {930--937},
  year    = {2010},
  doi     = {10.1109/TNN.2010.2045657}
}

@article{Ampazis,
  author  = {Ampazis, N. and Perantonis, S. J.},
  title   = {Two Highly Efficient Second-Order Algorithms for Training Feedforward Networks},
  journal = {IEEE Trans. Neural Netw.},
  volume  = {13},
  number  = {5},
  pages   = {1064--1074},
  year    = {2002},
  doi     = {10.1109/TNN.2002.1031939}
}

@article{KimLee,
  author={Kim, Cheol-Taek and Lee, Ju-Jang},
  journal={IEEE Transactions on Neural Networks}, 
  title={Training Two-Layered Feedforward Networks With Variable Projection Method}, 
  year={2008},
  volume={19},
  number={2},
  pages={371-375},
  doi={10.1109/TNN.2007.911739}
  }

@article{cybenko1989approximation,
  title={Approximation by superpositions of a sigmoidal function},
  author={Cybenko, George},
  journal={Mathematics of Control, Signals and Systems},
  volume={2},
  number={4},
  pages={303--314},
  year={1989},
  publisher={Springer}
}

@article{hornik1989multilayer,
  title={Multilayer feedforward networks are universal approximators},
  author={Hornik, Kurt and Stinchcombe, Maxwell and White, Halbert},
  journal={Neural Networks},
  volume={2},
  number={5},
  pages={359--366},
  year={1989},
  publisher={Elsevier}
}

@article{barron1993universal,
  title={Universal approximation bounds for superpositions of a sigmoidal function},
  author={Barron, Andrew R.},
  journal={IEEE Transactions on Information Theory},
  volume={39},
  number={3},
  pages={930--945},
  year={1993},
  publisher={IEEE}
}

@article{friedman1991mars,
  author  = {Friedman, Jerome H. and Grosse, Eric and Stuetzle, Werner},
  title   = {Multidimensional Additive Spline Approximation},
  journal = {SIAM J. Sci. Statist. Comput.},
  volume  = {4},
  number  = {2},
  pages   = {291--301},
  year    = {1983},
  doi     = {10.1137/0904023}
}

@inproceedings{Vinyals2012,
  title     = {Krylov Subspace Descent for Deep Learning},
  author    = {Vinyals, Oriol and Povey, Daniel},
  booktitle = {Proceedings of the Fifteenth International Conference on Artificial Intelligence and Statistics},
  series    = {Proceedings of Machine Learning Research},
  volume    = {22},
  pages     = {1261--1268},
  year      = {2012},
  publisher = {PMLR}
}

@inproceedings{Martens2010,
  author    = {Martens, James},
  title     = {Deep Learning via Hessian-Free Optimization},
  booktitle = {Proceedings of the 27th International Conference on Machine Learning},
  series    = {ICML'10},
  pages     = {735--742},
  year      = {2010},
  publisher = {Omnipress}
}

@article{yuan2009subspace,
  author  = {Yuan, Ya-xiang},
  title   = {Subspace methods for large scale nonlinear equations and nonlinear least squares},
  journal = {Optimization and Engineering},
  volume  = {10},
  number  = {2},
  pages   = {207--218},
  year    = {2009},
  doi     = {10.1007/s11081-008-9064-0}
}

@book{constantine2015active,
  author    = {Constantine, Paul G.},
  title     = {Active Subspaces: Emerging Ideas for Dimension Reduction in Parameter Studies},
  series    = {SIAM Spotlights},
  publisher = {SIAM},
  address   = {Philadelphia},
  year      = {2015},
  doi       = {10.1137/1.9781611973860}
}

@article{constantine2016accelerating,
  author  = {Constantine, Paul G. and Kent, Carson and Bui-Thanh, Tan},
  title   = {Accelerating Markov Chain Monte Carlo with Active Subspaces},
  journal = {SIAM Journal on Scientific Computing},
  volume  = {38},
  number  = {5},
  pages   = {A2779--A2805},
  year    = {2016},
  doi     = {10.1137/15M1042127}
}

\end{document}